\documentclass[11pt,draftcls,onecolumn]{IEEEtran}

\usepackage{amsthm}
\usepackage{amsfonts}
\usepackage{amsmath}
\usepackage{amssymb}
\usepackage{makeidx}
\usepackage{epsfig}
\usepackage{graphicx}
\usepackage{color}
\usepackage{cite}
\usepackage{url}

\usepackage[ruled,vlined]{algorithm2e}
\usepackage{cite} 	
\usepackage{url}

\usepackage[rgb,letterpaper]{xcolor}
\usepackage{algorithmic}
\usepackage{cite}
\usepackage[letterpaper]{psfrag}
\usepackage{tikz}
\usepackage[process=auto,crop=pdfcrop]{pstool}

\newtheorem{theorem}{Theorem}
\newtheorem{proposition}{Proposition}
\newtheorem{lemma}{Lemma}
\newtheorem{corollary}{Corollary}
\newtheorem{remark}{Remark}
\newtheorem{ex}{Example}
\newtheorem{definition}{Definition}

\newcommand{\switcheddelay }{{linear feedback system with switched delays}}

\newcommand{\re}{{\mathbb R}}

\newcommand{\n}{{\mathbb N}}
\newcommand{\ca}{{\cos{(\alpha)}}}
\newcommand{\sa}{{\sin{(\alpha)}}}

\newcommand{\spann}{{\mbox{span}}}

\newcommand{\cq}{{\mathbb{Q}}}

\providecommand{\com[2]}{\begin{tt}[#1: #2]\end{tt}}

\providecommand{\rmj[1]}{{\color{black}#1}}
\providecommand{\ad[1]}{{\color{black}#1}}

\ifCLASSINFOpdf
\else
\fi
\begin{document}

\title{Modeling, analysis and design of linear systems with switching delays
\author{Rapha\"el M. Jungers$^1$, Alessandro D'Innocenzo$^2$ and Maria D. Di Benedetto$^2$}
\thanks{$^1$ICTEAM Institute, Universit\'e catholique de Louvain, Louvain-la-Neuve, Belgium. Email: {\texttt{raphael.jungers@uclouvain.be}}}
\thanks{$^2$Department of Electrical and Information Engineering, Center of Excellence DEWS, University of L'Aquila, Italy. Email: {\texttt{alessandro.dinnocenzo@univaq.it}}, {\texttt{mariadomenica.dibenedetto@univaq.it}}}
\thanks{The research leading to these results has received funding from the European Union Seventh Framework Programme [FP7/2007-2013] under grant agreement n257462 HYCON2 Network of excellence. R.J. is supported by the Communaut\'e francaise de Belgique - Actions de Recherche Concert\'ees, and by the Belgian Programme on Interuniversity Attraction Poles initiated by the Belgian Federal Science Policy Office.  R.J. is a F.R.S.-FNRS Research Associate.}
}

\maketitle


\begin{abstract}
We consider the modeling, stability analysis and controller design problems for discrete-time LTI systems with state feedback, when the actuation signal is subject to switching propagation delays, due to e.g. the routing in a multi-hop communication network. We show how to model these systems as regular switching linear systems and, as a corollary, we provide an (exponential-time) algorithm for robust stability analysis.  \rmj{We also show} that the general stability analysis problem is NP-hard in general. Even though the systems studied here are inherently switching systems, we show that their particular structure allows for analytical understanding of the dynamics, and \rmj{even} efficient algorithms for some problems: for instance, we give an algorithm that computes in a finite number of steps the minimal look-ahead knowledge of the delays necessary to achieve controllability. We finally show that when the switching signal cannot be measured it can be necessary to use nonlinear controllers for stabilizing a linear plant.
\end{abstract}


\section{Introduction}\label{secIntro}

\begin{figure}[t]
\begin{center}
\includegraphics[width=0.5\textwidth]{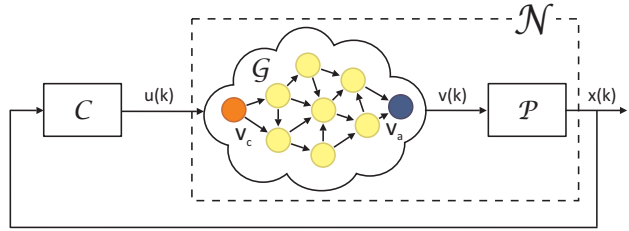}
\caption{State feedback control scheme of a multi-hop control network.}\label{figNetworkScheme}
\end{center}
\end{figure}

Wireless networked control systems are spatially distributed control systems where the communication between sensors, actuators, and computational units is supported by a wireless multi-hop communication network. The main motivation for studying such systems is the emerging use of wireless technologies for control systems (see e.g. \cite{akyildiz_wireless_2004}, \cite{SongIECON2010} and references therein) and the recent development of wireless industrial control protocols (such as WirelessHART and ISA-100). Although the use of wireless networked control systems offers many advantages with respect to wired architectures (e.g. flexible architectures, reduced installation/debugging/diagnostic/maintenance costs), their use is a challenge when one has to take into account the joint dynamics of the plant and of the communication protocol. Recently, a huge effort has been made in scientific research on Networked Control Systems (NCS), (see e.g.~\cite{Zhang2001},~\cite{WalshCSM2001},~\cite{SpecialIssueNCS2004},~\cite{MA_TAC2004},~\cite{HLA_TAC2009},~\cite{Andersson:CDC05},~\cite{MurrayTAC2009},~\cite{Hespanha2007},\cite{HeemelsTAC10}) and on the interaction between control systems and communication protocols (see e.g.~\cite{Astrom97j1},~\cite{walsh_stability_2002},~\cite{yook_trading_2002},~\cite{TabbaraCDC2007},~\cite{TabbaraTAC2007}). In general, the literature on NCSs addresses non--idealities (e.g. quantization errors, packets dropouts, variable sampling and delay, communication constraints) as aggregated network performance variables or disturbances, neglecting the dynamics introduced by the communication protocols. In~\cite{Andersson:CDC05}, a simulative environment of computer nodes and communication protocols interacting with the continuous-time dynamics of the real world is presented. To the best of our knowledge, the first integrated framework for analysis and co-design of network topology, scheduling, routing and control in a wireless multi-hop control network has been presented in \ad{\cite{AlurTAC11,PappasTAC2011,DInnocenzoTAC13}, where switching systems are used as a unifying formalism for control algorithms and communication protocols and sufficient conditions for stabilizing the plant are provided.} \ad{In particular, in \cite{DInnocenzoTAC13} the networked system is sampled at the time scale of the period of the communication protocol scheduling: this makes the system discrete-time linear time-invariant. In this paper we refine the model and consider a networked system sampled at the more accurate time scale of the transmission slots of each communication node: this makes the system discrete-time switching linear, which is a much more difficult mathematical framework. The aim of this paper is providing exact necessary and sufficient conditions for stability analysis and controllability.}

We assume in our model that a multi-hop network $\mathcal G$ provides the interconnection between a state-feedback discrete-time controller $\mathcal C$ and a discrete-time LTI plant $\mathcal P$ (see Figure \ref{figNetworkScheme}). The network $\mathcal G$ consists of an acyclic graph where the node $v_c$ is directly connected to the controller and the node $v_u$ is directly connected to the actuator of the plant. As classically done in (wireless) multi-hop networks to improve robustness with respect to node failures we assume that the number of paths that interconnect $\mathcal C$ to $\mathcal P$ is greater than one and that for any actuation data sent from the controller to the plant a unique path is chosen. Each path is characterized by a delay in forwarding the data (see \cite{DInnocenzoTAC13} for details), as a consequence each actuation data will be delayed of a finite number of time steps according to the chosen routing. Since this choice usually depends on the internal status of the network, i.e. because of node and/or link failures, we consider the choice of the routing path as an external disturbance and address both the cases when it is measurable or not: as a consequence our system is characterized by switching time-varying delays of the input signal.

Systems with time-varying delays have attracted increasing attention in recent years (see e.g. \cite{LiuCTA2006}, \cite{HetelCDC2007}, \cite{HeemelsTAC2010}, \cite{ShaoTAC2011}, \cite{HetelIJC2012} and references therein). In \cite{HetelSCL2011} it is assumed that the time-varying delay is approximatively known and numerical methods are proposed to exploit this partial information for adapting the control law in real time. \ad{Our modeling choice is close to the framework in \cite{HetelCDC2007}. However, in that work the delay is determined when the actuator receives the control packet rather than, as we assume in this paper, when the controller emits a control packet.} \rmj{This makes our model slighlty differ from these settings.  It is in our view more realistic because} due to the routing, control commands generated at different times can reach the actuator simultaneously, their arrival time can be inverted, and it is even possible that at certain times no control commands arrive to the actuator.\\ \rmj{The} LMI based design procedures \rmj{that} have been developed for switching systems with time-varying delays (see e.g. \cite{HetelTAC2006} and \cite{ZhangCTA2008}) \rmj{do not take into account} the specific structure of the systems induced by the fact that the switching is restricted on the delay-part of the dynamics.  Our goal is to leverage this particular structure in order to improve our theoretical understanding of the dynamics at stake in these systems.  As we will see, it enables us to design tailored controllers whose performances or guarantees are better than for classical switching systems.

As a first contribution in this paper we prove that our networked systems can be modeled by pure switching systems where the switching matrices assume a particular form. As a byproduct we provide new LMI stability conditions with arbitrarily small conservativeness (which can be fixed a priori before computation)\rmj{.  We also} show that the general stability analysis problem is NP-hard in general.
%

As a second contribution, we address the controller design problem by assuming that, for each time $t$, the controller is aware \ad{of the propagation delays of the actuation signals sent at times $t, t+1, \ldots, t+N-1$. We call $N$ the \emph{look-ahead} parameter. If $N = 0$ the controller is not aware of any of the past, current and future propagation delays: we will call this situation the \emph{delay independent case}. If $N \geq 1$ the controller is aware of the current and $N-1$ next future routing path choices, and keeps memory of the past delays: we will call this situation the \emph{delay dependent case}. Note that $N = 1$ describes the situation where the controller is only aware of the propagation delay at the current time $t.$}

From the network point of view the practically admissible values for $N$ depend on the protocol used to route data (see \cite{YangWiMesh2011} and references therein for an overview on routing protocols for wireless multi-hop networks). If the controller node $v_c$ of $\mathcal G$ is allowed by the protocol to chose a priori the routing path (e.g. source routing protocols), then we can assume that the controller is aware of the routing path and the associated delay, and therefore is also aware of the switching signal (i.e. $N>0$). If instead the protocol allows each communication node to choose the next destination node according to the local neighboring network status information (e.g. hop-by-hop routing protocols) then we cannot assume that the controller is aware of the routing path, and therefore of any of the past, current and future propagation delays (i.e. $N=0$).

We will first analyze the situation where one can chose an arbitrarily large (but finite) look-ahead parameter $N$. This can occur in situations where the dynamics of the plant is given, but one can design the protocol, and thus require a certain look-ahead knowledge if needed. We prove that in this case the controllability verification problem can be split into two sub-problems, one characterized by a regular matrix and one by a nihilpotent matrix: we provide an exponential time algorithm to solve the regular case and a polynomial time algorithm to solve the nihilpotent case.
Finally, we analyze the situation where the look-ahead parameter $N$ is fixed and given as part of the problem. We show that in this case it is much harder to decide controllability of the plant: indeed we show that it can be necessary to make use of nonlinear controllers for stabilizing a linear plant.

The paper is organized as follows. In Section \ref{secModeling} we provide the problem formulation and the modeling framework, and define the two families of controllers introduced above.
In Sections \ref{secAnalysis}, \ref{secDesignDD} and \ref{secDesignDI} we address the stability verification and the controller design problems, both for the \emph{delay dependent} and the \emph{delay independent} case.
In Section \ref{secConclusions} we provide concluding remarks and open problems for future research.


\section{Modeling}\label{secModeling}

In this paper we will address the problem of stabilizing a discrete-time LTI system of the form
$$
x(t+1) = A x(t) + B u(t),\quad y(t) = x(t),\quad t \geq 0,
$$
with $A \in \mathbb R^{n \times n}$ and $B \in \mathbb R^{n \times m}$, using a state-feedback controller $\mathcal C$. From now on we will obviously assume that the plant pair $(A,B)$ is controllable. We assume that the control signal $v(t)$ generated by $\mathcal C$ is relayed to the actuator of the plant $\mathcal P$ via a multi-hop network \cite{AlurTAC11}. The network $\mathcal G$ consists of an acyclic graph $(V,E)$, where the node $v_c \in V$ is directly interconnected to the controller $\mathcal C$ and the node $v_u$ is directly interconnected to the actuator of the plant $\mathcal P$. In order to relay each actuation data $v(t)$ to the plant, at each time step $t$ a unique path of nodes that starts from $v_c$ and terminates in $v_u$ is exploited. As classically done in (wireless) multi-hop networks to improve robustness of the system with respect to node failures we exploit redundancy of routing paths, therefore the number of paths that can be used to reach $v_u$ from $v_c$ is assumed to be greater than one. To each path a different delay can be associated in transmitting data from $v_c$ to $v_u$, depending on the transmission scheduling and on the number of hops to reach the actuator (see \cite{DInnocenzoTAC13} for details). Since the choice of the routing path usually depends on the internal status of the network (e.g. because of node and/or link failures, bandwidth constraints, security issues, etc.), we assume that the chosen routing path is time-varying. As a consequence the control signal $v(t)$ at time $t$ will be delayed of a finite number of time steps that we model as a disturbance signal $\sigma(t)\in D: \,t \geq 0$, where $D \subseteq \{0, 1, \dots, d_{max}\}$ is the set of possible delays introduced by all routing paths and $d_{max}$ is the maximum delay. For the reasons above we model the dynamics of the networked control system $\mathcal N$ as follows:

\begin{definition}\label{def-general-sdsystem}
The dynamics of the interconnected system $\mathcal N$ can be modeled as\rmj{\footnote{In this paper we adopt the \lq Matlab notation,\rq{}  that is, $v(t - d_{max}:t), \sigma(t - d_{max}:t)$ represent the latest {$d_{max}+1$} values of the output of the controller $v(\cdot) \in \mathbb R^{m}$ and of the switching signal $\sigma(\cdot) \in D$.}}
\begin{align}\label{sdsystem}
x(t+1) = A x(t) + B u(v(t - d_{max}:t),\sigma(t - d_{max}:t)) = A x(t) + B \sum_{
\tiny
\begin{array}{cc}
    t - d_{max} \leq t' \leq t: \\
    t'+\sigma(t')=t
\end{array}}
\normalsize
v(t').
\end{align}
 We \rmj{also} define the signal of {\emph{actuation times}} $\tau : \mathbb N \rightarrow \{0,1\}$ such that $\tau(t) = 1$ if there exists $t' \leq t$ with $t = t' + \sigma(t')$, $\tau(t) = 0$ otherwise.
\end{definition}
\begin{proposition}
The dynamical system induced by Equation \eqref{sdsystem} is equivalent to the following discrete-time linear switching system:
\begin{equation}
x_e(t+1) = A_e x_e(t) + B_e(\sigma(t)) v(t),
\end{equation}
where
$x_e(t) \doteq (x(t),u_{1}(t),u_{2}(t),\dots, u_{d_{max}}(t)) \in \mathbb R^{n + m d_{max}}$ represents the internal state of $\mathcal N$, with $u_d(t) \in \mathbb R^{m}$, $\forall d \in \{1, \ldots, d_{max}\}$ the actuation signal that is forecast to be applied to the plant at time $t+d-1$, i.e.
\begin{equation}\label{eq-ud}
u_{d}(t)=\sum_{t'<t: t'+\sigma(t')=t+d-1} v(t'),
\end{equation}
and where
\begin{align*}
A_e = \begin{pmatrix}
A&B&0&\dots & 0\\
0&0&I&\dots & 0\\
\vdots&\vdots&\vdots&\ddots&\vdots\\
0&0&0&\dots &I\\
0&0&0&\dots &0
\end{pmatrix},\quad\quad
B_e(\sigma(k)) = \left(
                 \begin{array}{l}
                    B \cdot \delta_{\sigma(k),0}\\
                    I \cdot \delta_{\sigma(k),1}\\
                   \vdots \\
                    I \cdot \delta_{\sigma(k),d_{max}-1}\\
                    I \cdot \delta_{\sigma(k),d_{max} }
                 \end{array}
               \right),
\end{align*}
with $\delta_{\sigma(k),d}$, $d \in \{0, 1, \ldots, d_{max}\}$, the Kronecker's delta.
\end{proposition}
The above model is quite general and allows representing a wide range of routing communication protocols for (wireless) multi-hop networks \cite{YangWiMesh2011}. We remark that several variations are possible. For instance, in our setting, it could happen during the run of {the system} that at some particular time no feedback signal comes back to the plant: we assume that in this situation the actuation input to the plant is set to zero. A variation is to implement a \emph{hold} that would keep memory of the previous input signal and resend it if the new one is empty. Moreover, in our setting, it could also happen that two control signals sent at different times reach the actuator simultaneously: we assume that in this situation the actuation input to the plant is set to the sum of the control signals that arrive simultaneously. A variation is to keep the most recent control signal and discard all the others.

We defer the comparison of such variants for further studies {(see \cite{GommansIJC2013} for a recent work that \rmj{takes into account packet dropouts and provides a comparison among some of these approaches in this setting}).}

Before introducing the delay dependent and delay independent classes of linear controllers for the feedback networked scheme we formally define the state space of the controller: for classical (i.e. non-switched) feedback systems with fixed delay it is well known that the system can be neither controllable nor stabilizable if the feedback only depends on $x(t),$ that is if the controller does not have a memory of its past outputs. Therefore, \emph{we allow that the controller keeps memory of its past $d_{max}$ outputs $v(t-d_{max}), \ldots, v(t-1)$}. Also, we define $N$ to be the \emph{look-ahead} of the controller.  That is, for each time $t,$ the controller is aware of the $N > 0$ future routing path choices, and therefore of the future propagation delays $\sigma(t), \sigma(t+1), \ldots, \sigma(t+N-1).$  As illustrated in Section \ref{secIntro} the admissible range of values for $N$ depends on the protocol used to route data in the network. If $N \geq 1$ we also assume that the controller keeps memory of the past $d_{max}$ switching signals $\sigma(t-d_{max}), \ldots, \sigma(t-1)$. If instead $N = 0$ then we assume that the controller ignores the past, current and future switching signals.


\subsection{Delay dependent case} \label{subsecModelingDD}

In the delay-dependent case, (i.e. $N>0$), the controller can reconstruct the state $x_e(t)$ of $\mathcal N$ via a linear combination \rmj{of its past outputs $v(t'),$ $t'<t$} as in Equation \eqref{sdsystem}. Note that, once the current state of $\mathcal N$ has been reconstructed, the past $d_{max}$ switching signals $(\sigma(t-d_{max}) : \sigma(t-1))$ are irrelevant for the controller design.  Thus, we can assume without loss of generality that the controller only depends on the signals $(\sigma(t) : \sigma(t+N-1))$.

\begin{definition}\label{def-dep} Assume that at each time the sequence of the next $N > 0$ switching signals is known. We define \emph{a delay-dependent switching linear control law with look-ahead $N$} as follows
\begin{equation}\label{eq-vt-dep}
v(t) = K(\boldsymbol{\sigma}) x_e(t),
\end{equation}
with $\boldsymbol{\sigma} = (\sigma(t), \cdots, \sigma(t+N-1)) \in D^{N}$ and $K(\boldsymbol{\sigma}) \in \mathbb R^{m\times(n + m d_{max})}$.
\end{definition}

We call the controller of {Definition \ref{def-dep} \emph{acausal} if $N > 1.$ Note that in several practical situations the networking protocol can be designed to chose at any time $t$ the future routing paths up to $t+N-1.$ If it is not possible to know the signal $\sigma(s)$ for $s>t$ (i.e. $N=0$ or $1$) we say that we have a \emph{causal} controller. } The dynamics of the closed loop system can be written as follows:
\begin{proposition} \label{prop-delay-dependent-cl}
Given the system \eqref{sdsystem} and a switching linear control law as in Definition \ref{def-dep}, the closed loop system can be modeled as follows:
\begin{equation}\label{switched-closed-loop-dd}
x_e({t+1}) = M(\sigma(t)) x_e(t), \quad M(\sigma(t)) \in \Sigma := \{A_e + B_e(\sigma(t)) K(\boldsymbol{\sigma}) : \boldsymbol{\sigma} \in D^{N} \}.
\end{equation}
\end{proposition}


\subsection{Delay independent case ($N=0$)} \label{subsecModelingDI}

In Definition \ref{def-dep} we assumed that the controller knows the previous values of the switching signal, and thus can reconstruct $x_e(t)$ by applying equation (\ref{eq-ud}). If $N = 0$ this is not possible since the controller ignores the current and thus the previous values of the switching signal, and the only variables it can use are $x(t)$ and the $d_{max}$ past control commands $v(t-d_{max}),\dots,v(t-1)$. For this reason, we define
\begin{equation}
v_e(t) \doteq (x(t), v(t-d_{max}),\dots,v(t-1))
\end{equation}
the state variable accessible to the controller in the delay independent case.
\begin{definition}\label{def-indep} Assume that at each time the switching signal is unknown, i.e. $N = 0$. We define a static control law as follows:
\begin{equation}\label{eq-vt}
v(t)= K v_e(t), \quad K \in \mathbb R^{m\times(n + m d_{max})}.
\end{equation}
\end{definition}

The dynamics of the closed loop system can be written as follows:
\begin{proposition} \label{prop-delay-independent-cl}
Given the system \eqref{sdsystem} and a linear control law as in Definition \ref{def-indep}, the closed loop system can be modeled as follows:
\begin{equation}\label{switched-closed-loop-independent}
v_e({t+1})= M(\sigma,t) v_e(t),
\end{equation}
where
\begin{eqnarray}\label{eq-sigma-cl-independent}
M(\sigma,t)=\left\{
\begin{pmatrix}
    A & B \cdot \gamma_{d_{max}}(t) & B \cdot \gamma_{d_{max}-1}(t) &\cdots & B \cdot \gamma_{1}(t) \\
    0&0&I&\cdots & 0\\
    \vdots&\vdots&\vdots&\ddots&\vdots\\
    0&0&0&\cdots&I\\
    K_0&K_1&K_2 &\cdots&K_{d_{max}}
\end{pmatrix} +
\begin{pmatrix}
    B K \cdot \gamma_0(t)\\
    0\\
    \vdots\\
    0\\
    0
\end{pmatrix}\right\}
\end{eqnarray}
with $K = [K_0\ K_1\ K_2\ \cdots \ K_{d_{max}}]$ and $\gamma_{d}(t) = 1$ if $\sigma(t-d)=d$, $\gamma_d(t) = 0$ otherwise.
\end{proposition}
There are two important differences between the systems defined by Equations \eqref{switched-closed-loop-dd} and \eqref{switched-closed-loop-independent}. First, the set $\Sigma$ in (\ref{switched-closed-loop-independent}) has a number of matrices that can be exponential in the number of delays $|D|$ because the matrix $M(\sigma,t)$ depends on the values $(\sigma(t),\sigma(t-1),\dots,\sigma(t-d_{max}))$. Second, for the same reason, the closed-loop formulation (\ref{switched-closed-loop-independent}) is not a switching system with arbitrary switching signal, as successive occurrences of $M(\sigma,t)$ are correlated. Because of this correlation in the succession of matrices this setting seems harder to be represented as a pure switching system. Even though recent methods based on LMI criteria have been proposed that offer a natural framework for analyzing switching signals described by a regular language (e.g. \cite{ajprhscc11,ajprhscc12,LeeD06}), it would be convenient to have a formulation of the closed loop switching system without any constraint on the switching signal. Indeed, more methods for analyzing switching systems have been designed in the general framework of unconstrained switching. In the following theorem we show that it is always possible to model system \eqref{switched-closed-loop-independent} as a pure switching system.  Moreover, \rmj{even though we slightly augment the state space, the whole system we obtain has a polynomial size with respect to the initial size of the problem}.
\begin{theorem}\label{thm-ssreduction-independent}

Any $n$-dimensional \switcheddelay{} and delay-independent control of dimension $m$ and set of delays $D$ can be represented as a switching system with arbitrary switches among $|D|$ matrices, characterized by a $(n+2d_{max}m)$-dimensional state space.
\end{theorem}
\begin{proof}
The main idea of the proof is to make use of \emph{both $x_e(t) = (x(t),u_{1}(t),u_{2}(t),\dots, u_{d_{max}}(t))$ and $v_e(t) \doteq (x(t), v(t-d_{max}),\dots,v(t-1))$} in the closed loop state-space representation of the system. Recall that $u_d(t)$ is the sum of the previous outputs of the controller, that are forecast to arrive at the plant at time $t+d, d \in D$ and that $v(t)$ is the output of the controller at time $t$. Of course, the controller does not know $u_d(t)$ in this delay-independent setting (we will take that into account in the construction), but this variable is needed in order to reconstruct the feedback signal. On the other hand, $v(t)$ is needed in order to represent the memory of the controller.  We now formally describe the state-space, and then the matrices. Let
$$
K=(K_0,K_1,\dots,K_{d_{max}})\in \re^{m\times (n+d_{max} m)}
$$
be the linear controller and
$$
V(t)=(v(t-d_{max}),\dots,v(t-1))^T\in\re^{d_{max} m}
$$
the memory of the controller. We define a state-space vector $w(t)\in \re^{n+2d_{max}m}$ as follows:
$$
w (t)=(x(t),u_1(t),\dots, u_{d_{max}}(t),v(t-d_{max}),\dots,v(t-1)).
$$
For any $\sigma(t) \in D$, the following equations describe the switching linear dynamics of $w(t)$:
\begin{eqnarray}
\nonumber x(t+1)&=&Ax(t)+B(u_1(t) + z_{0}(\sigma(t)),\\
\nonumber u_s(t+1)&=&u_{s+1}(t) + z_{s}(\sigma(t)), \quad 1\leq s\leq d_{max} \\
\nonumber V(t+1)&=& (0,\dots,0,K_0x(t))^T + \begin{pmatrix}0&I&\dots&0\\0&0&\ddots&0\\0&\dots&0&I\\K_1 &\dots &&K_{d_{max}}\end{pmatrix}V(t).
\end{eqnarray}
In the above equations $z_s(\sigma(t))$, $s=0, \dots, d_{max}$, represents the controller output to be fed back to the plant with a delay $s = \sigma(t)$:
\begin{equation*}
z_s(\sigma(t)) = \left\{ \begin{array}{ll}
                   K\begin{pmatrix}x(t)\\V(t) \end{pmatrix} &  \mbox{ if } \sigma(t) = s, \\
                   0 & \mbox{otherwise}.
                 \end{array}\right.
\end{equation*}
\end{proof}

As a consequence all tools developed for general switching systems can be used for stability analysis and controller design (e.g. \cite{jungers_lncis,sun-ge}). However, the particular delay model that we are considering makes our system a special case of general switching systems, endowed with a characteristic matrix structure. In the next sections we aim at exploiting such special structure to derive tailored stability analysis and controller design results beyond the theoretical barriers that hold for general switching systems.

\section{Stability analysis} \label{secAnalysis}

We first define the stability notion for System \eqref{sdsystem} with respect to delay dependent and delay independent control laws.
\begin{definition}
We say that a system \eqref{sdsystem} is \emph{stable} for a given control law as in Definition \ref{def-dep} (resp. Definition \ref{def-indep}) if, for any switching signal $\sigma(t)$ and for any initial condition $x_e(0)$,
\begin{equation*}
\lim_{t\rightarrow \infty} x_e(t)=0 \quad (\textrm{resp. } \lim_{t\rightarrow \infty} v_e(t)=0).
\end{equation*}
\end{definition}

As illustrated above, a \switcheddelay{} can be put in the well studied framework of linear switching systems with arbitrary switching signal. Even though these systems have been at the center of a huge research effort in the last decades (see for instance \cite{sun-ge,LeeD06,jungers_lncis,valcher-positive}), they are known to be very difficult to handle. Nevertheless, it follows from Proposition \ref{prop-delay-dependent-cl} and Theorem \ref{thm-ssreduction-independent} that one can check the stability of a given \switcheddelay{} with arbitrary precision:
\begin{corollary}\label{cor-approx-dep}
Given the system \eqref{sdsystem} and a control law as in Definition \ref{def-dep} (resp. Definition \ref{def-indep}), and for any $\epsilon > 0$, there exists an algorithm that computes in finite time the worst rate of growth of the system up to an error of $\epsilon$. More precisely, for any real $r>0$ the algorithm decides whether
\begin{itemize}
\item $\exists K \in \re : \forall \sigma, \forall t,\ |x_e(t)|\leq K (r+\epsilon)^t$ $\big($resp. $|v_e(t)|\leq K (r(1+\epsilon))^t$$\big)$;
\item $\exists K \in \re, K>0, \exists \sigma:\ |x_e(t)|\geq K (r-\epsilon)^t$ $\big($resp. $|\rmj{v_e(t)}|\geq K (r(1-\epsilon))^t$$\big).$
\end{itemize}
\rmj{Moreover there exists such an algorithm that terminates in an amount of time which is \ad{a polynomial $P_\epsilon(n^{\log(|D|)})$ of degree ${\log(|D|)},$} where $n$ is the dimension of the plant and $|D|$ is the number of delays.}
\end{corollary}
\begin{proof}
Proposition \ref{prop-delay-dependent-cl} reformulates the \switcheddelay{} as a classical switching system, thus it is possible to apply one of the classical stability decision procedures derived in \cite[Corollary 3.1]{protasov-jungers-blondel09} or \cite[Theorem 2.16]{jungers_lncis}.  \rmj{These procedures are known to terminate within a time bounded by a polynomial in $n^{\log{(m)}},$ where $n$ is the dimension of the matrices, and $m$ is the number of matrices.}
\end{proof}

The above corollary provides a tool to approximate the worst rate of growth, by bisection on $r$. Thus, it is possible to decide with an arbitrary precision whether a \switcheddelay{} endowed with a delay dependent or delay independent control law is stable. However, \rmj{the complexity of the algorithm (the polynomial mentioned in the above corollary) strongly depends on the accuracy $\epsilon$ (this polynomial becomes huge for small values of $\epsilon$) and no algorithm is known that works in polynomial time with respect to $\epsilon.$ Hence, our solution does not work in polynomial time with respect to the accuracy}. This is not surprising in view of our next result: we show that given a system with variable delays and its controller, it is in general NP-hard to decide whether the controller asymptotically stabilizes the system.

%

\begin{theorem}
Given the system \eqref{sdsystem} and a delay dependent switching linear control law as in Definition \ref{def-dep}, unless $P=NP$, there is no polynomial-time algorithm that decides whether the corresponding closed loop system as in Proposition \ref{prop-delay-dependent-cl} is stable.
Also, the question of whether the system remains bounded is Turing-undecidable.
This is true even if the matrices have nonnegative rational entries, and the set of delays is $\{0,1\}.$
\end{theorem}
\begin{proof}
Our proof works by reduction from the matrix semigroup stability, and the matrix semigroup boundedness, which are well known to be respectively NP-hard and Turing undecidable \cite[Theorem 2.4 and Theorem 2.6]{jungers_lncis}. In this problem, one is given a set of two matrices $\Sigma=\{A_1,A_2\}\subset \cq_+^{n\times n}$ ($\cq_+$ is the set of nonnegative rational numbers) and one is asked whether for any sequence $(i_t)_0^\infty,\ i_t\in [1,2],$ the corresponding product $A_{i_1}A_{i_2}\dots A_{i_T}$ converges to the zero matrix when $T\rightarrow \infty$ (respectively remains bounded when $T\rightarrow \infty$).

Let us consider a particular instance $\Sigma=\{A_1,A_2\}\in \cq_+^{n\times n}$ of the matrix semigroup stability (resp. boundedness) problem. We will build a closed loop system as follows:
\begin{eqnarray}
x_e({t+1})&=&M_i x_e(t),\quad M_i\in \Sigma',
\end{eqnarray}
where $\Sigma'$ is a set of $2n \times 2n$ matrices, and prove that the set $\Sigma'$ is stable (resp. product-bounded) if and only if $\Sigma$ is.
Our construction is as follows: we set $D=\{0,1\}$ as the set of delays and $N=1$ for the look-ahead, and we build a \switcheddelay{} with a plant characterized by internal state space dimension $n$ and input space dimension $m=n$ as follows: the system matrix is given by $ A=0,\, B = I$, and the feedback matrix, assuming that $N=1$, in block form given by $K(0)= (A_1 \ 0 )$ for $d=0$ and $K(1)= (A_2 \ 0)$ for $d=1$. Thus, the corresponding closed loop feedback switching system can be expressed from Proposition \ref{prop-delay-dependent-cl} as $$ x_e({t+1})= M(\sigma(t)) x_e(t), \quad M(\sigma(t)) \in \Sigma', $$
where \begin{equation}\label{eq-np-switching}\Sigma'= \left\{ \begin{pmatrix} A_1 & I \\0&0\end{pmatrix},\begin{pmatrix}0&I\\ A_2 &0 \end{pmatrix}\right\}.\end{equation}

Writing $x_e(t) = (x(t), u_1(t))$ we have that, depending on $\sigma(t)$, either $x_e({t+1})= (A_1x(t) + u_1(t),0)$ or $x_e({t+1})= (u_1(t), A_2x(t))$. \rmj{From this, it is straightforward to see that the set $\Sigma'$ is stable (resp. product bounded) if and only if $\Sigma$ is.} Indeed, the blocks in the products of matrices in $\Sigma'$ are arbitrary products of matrices in $\Sigma.$  This concludes the proof.
\end{proof}

\begin{remark}
It is not known (to the best of our knowledge) whether the matrix semigroup stability problem is Turing decidable (say, for matrices with rational entries). Thus, the above proof does not allow us to conclude that the \switcheddelay{} stability problem is undecidable. This is why we only claim that the stability problem is NP-hard, while the boundedness problem is provably Turing undecidable.
\end{remark}



\section{\ad{Controller design with arbitrary look-ahead: exact algorithms}} \label{secDesignDD}

In this section we address the following problem: given a system with variable delays as in \eqref{sdsystem}, design a controller \ad{with \emph{arbitrarily large} look-ahead} such that the closed loop system is stable (by \emph{arbitrarily large} we mean that $N$ is not part of the problem, but rather, one is allowed to chose a suitable value for it). This problem is challenging since very little is known in the literature about the design of switching systems. In this section, for the sake of clarity, we restrict ourself to the single input case, that is, $m=1.$ We first define the controllability notion for System \eqref{sdsystem}.


\begin{definition}\label{defControllability}
We say that the system \eqref{sdsystem} is controllable with look-ahead $N$ if, for any initial state $x_0$, any final state $x_f$ and any switching signal $\sigma$, there exists a control signal $v(t, x(t-d_{max}:t), \sigma(t-d_{max}:t+N))$ such that
$$
\exists t \geq 0: x(t)=x_f.
$$
\end{definition}


The existence of a controller as in Definition \ref{defControllability} seems to be very hard to decide: to the best of our knowledge, controller design is widely overlooked in the literature on general switching systems, and only sufficient conditions for the existence of a \emph{linear} controller are known \cite{LeeD06,sun-ge}. As a first attempt to tackle the problem for our systems with varying delays, we now assume that the controller has an arbitrarily large look-ahead knowledge of the switching signal. In this case we are able to decide controllability, and efficiently build a controller. It turns out that one can compute a finite value $N$ depending only on the dimension of the plant and the set of delays such that if the system is controllable with infinite look-ahead, it is controllable with finite look-ahead $N.$  This obviously gives conditions for controllability when the look-ahead is a fixed finite value $N':$ if the system is uncontrollable with infinite look-ahead, it is clearly uncontrollable with the actual value $N';$ if, on the other hand, the system is controllable with infinite look-ahead and $N'\geq N,$ then the actual system is controllable.  \rmj{In order to handle the fact that the look-ahead is arbitrarily large}, in the next definition, we introduce a controllability matrix as if the controller knew at time $t=0$ the infinite future sequence of signals $\sigma(t).$
\begin{definition}\label{defControllabilityMatrix}
Given a system \eqref{sdsystem} and a switching signal $\sigma(t)$ we define
$$
C_t(A,B,D,\sigma(t))
$$
the \emph{controllability matrix at time $t$}, whose columns are given by
\begin{equation}\label{eq-controllability-matrix}
\{A^{t-t'-\sigma(t')}B:\, t'\geq 0,\ t-t'-\sigma(t')\geq 0\}.
\end{equation}
The order of the columns (although not that important) is by increasing order of $t',$ so that
\begin{equation}\label{eqControllability}
x_t=A^t x_0+C_t v_t,
\end{equation}
where the components of $v_t$ are given by
$$
v_t(t')=v(t'):\, t'\geq 0,\ t-t'-\sigma(t')\geq 0,
$$
namely by all the control signals delivered to the actuator up to time $t.$
\end{definition}

By an abuse of notation, we denote by $\spann(C_t)$ the space generated by the columns of the matrix $C_t$.

\begin{proposition}\label{propContrCond}
The System \eqref{sdsystem} is uncontrollable if and only if there exists a switching signal $\sigma(t)$ such that, for all $t \in \mathbb N$, we have
\begin{equation}\label{eqContrCond}
\spann(C_t)\neq \re^n.
\end{equation}
\end{proposition}



We first completely solve the design problem for one-dimensional systems with $n=m=1$, namely with $A = a \in \mathbb R$ and $B = b \in \mathbb R$. Since we assumed that $(A,B)$ is controllable, then $b \neq 0$. For a classical LTI system with fixed delay $d$, in the case where the controller has a memory of its past $d_{max}$ outputs, a solution that drives the trajectory onto the origin in finite time is given by an extension of the Ackermann formula:
\begin{equation}\label{eq-ack}
K^*(d) = (-a^{d+1}/b,-a^{d},-a^{d-1},\dots,-a).
\end{equation}
It turns out that for a \switcheddelay{} too, there is always a solution that reaches the origin in at most $d_{max}$ steps using a controller as in Definition \ref{def-dep} that \rmj{only requires the knowledge} of the current switching signal, i.e. with $N = 1$:
\begin{theorem}
Given system (\ref{sdsystem}) let $n=m=1$ and $a \neq 0$: then the system reaches the origin at latest at time $t=d_{max}+1$ using the following switching linear controller:
$$
K(d)=K^*(d_{max}) \cdot a^{d-d_{max}},
$$
with $K^*(d_{max})$ as in equation \eqref{eq-ack}.
\end{theorem}
\begin{proof}
Let $x_e(t) = (x(t),u_{1}(t),u_{2}(t),\dots, u_{d_{max}}(t))$ be the state of the system. From Equation (\ref{eq-ack}) we have, at any time $t,$
\begin{eqnarray}\label{eq-ackerman-zero} \nonumber(a^{d_{max}+1}/b)x(t) + \sum_{s=1}^{d_{max}} a^{d_{max}-s+1}u_{s}(t) + K(\sigma(t))x_e(t)(a^{d_{max}-{\sigma(t)}})&=&0. \end{eqnarray}
By Definitions \ref{def-dep} and \ref{sdsystem} it follows that for $s=1,\dots,d_{max},$
\begin{eqnarray}\label{eq-us}
u_s(t+1)&=&u_{s+1}(t)\quad \mbox{if }\sigma(t)\neq s\\
\nonumber u_s(t+1)&=&u_{s+1}(t)+v(t)\quad \mbox{if }\sigma(t)= s,
\end{eqnarray}
where we fix for conciseness of notations that $u_{d_{max}+1}=0.$ Observe that the last term in the left-hand side of Equation (\ref{eq-ackerman-zero}) is equal to $v(t)(a^{d_{max}-{\sigma(t)}}).$ Multiplying that equation by $a,$ and making use of Equation \eqref{eq-us}, we obtain:
\begin{align*}
0 =& a^{d_{max}+1}/b(ax(t)+bu_1(t)+bz(t))+\sum_1^{d_{max}}{a^{d_{max}+1-s}u_s(t+1)}\\
=& (a^{d_{max}+1}/b)x({t+1})+\sum_1^{d_{max}}{a^{d_{max}+1-s}u_s(t+1)} = K^*(d_{max})x_e(t+1) = v({t+1})\ a^{d_{max}-{\sigma(t+1)}},
\end{align*}
where, again for conciseness, we introduce the variable $z$ such that $z(t)=v(t)$ if $0 \in D$ and $\sigma(t)=0$, and $z(t)=0$ otherwise. In conclusion, if the controller is applied at time $1,$ the output of the controller at time $2$ is $v_{2}=0.$ Thus, by induction, $\forall t'>1,\, v({t'})=0$. This implies (see Equation (\ref{eq-ud})) that $\forall t''>d_{max},\forall s,\,u_s{(t'')}=0.$ In turn, since $v_{d_{max+1}}= K^*(\sigma(d_{max}+1))x_e(d_{max}+1)=0,$ this implies that $x_{d_{max}+1}=0.$
\end{proof}
The following example shows that the design problem is not trivial as soon as the dimension $n$ of the plant state is equal to $2$ even if $m=1$.
\begin{ex}
Consider a \switcheddelay{} with the following parameters:
$$
A=\begin{pmatrix}0& 2\\2& 0\end{pmatrix},\ b=\begin{pmatrix}0& 1\end{pmatrix}^T,\ D=\{0,1\},\ \sigma(t)=t \mbox{ mod } 2.
$$
That is, $\sigma(t) = 0$ when $t$ is even, and $1$ when $t$ is odd. One can show by induction that, if $x(0)=(1, 0)$, for any even time $t$, $x_1(t)=2^t$. As a consequence the system is uncontrollable even though the pair $A,b$ is controllable.
\end{ex}

Motivated by the example above, we investigate in the following theorems the controllability property for System (\ref{sdsystem}):

\begin{definition}
We say that a matrix $A\in \re^{n\times n}$ is a \emph{block \rmj{cyclic} permutation of order $p$} if there exists a block-partition of the entries such that $A$ acts as a cyclic permutation on these blocks:
\small
\begin{align*}
\exists n_0=0< n_1<\dots <n_p= n : \left(\forall 1\leq i\leq j \leq n,\, A_{i,j}\neq 0 \Rightarrow \exists s: n_{s-1}< i\leq n_{s}<j\leq n_{s+1}\mbox{ or } i>n_{p-1},j \leq n_1\right).
\end{align*}
\normalsize
We say that $B\in \re^{n\times m}$ has a \emph{zero-block of index $n_k$} (w.r.t. the blocks defined above) if
$$
\forall n_{k-1}<i\leq n_k, \ 1\leq j\leq m,\; B_{i,j}=0.
$$
\end{definition}

\begin{theorem}\label{thm-block-permut}
Suppose that $A$ can be put in a block cyclic permutation form in some basis, and $b$ has one or more zero blocks in the same basis.  Let us denote by $Z$ the set  of indices of the zero blocks in $b.$  Then, the system is uncontrollable provided that \begin{equation}
\forall x\in [1,p],\, \exists z\in Z, d\in D: \, x= z-d\, \mbox{mod }p,
\label{eq-thm-blocks}
\end{equation} where $p$ is the order of the block-permutation.
\end{theorem}
\begin{proof}
We simply present a sequence of delays that satisfies the property that at any time $t,$ all the vectors in the controllability matrix have a zero block at index $t+s \  \mbox{mod } p,$ where $s\in [1,p]$ can be chosen arbitrarily. At any time $t,$ define $\sigma(t)$ such that $z-\sigma(t)=s+t\  \mbox{mod } p$, for some $z\in Z.$  Thus, the vector $b$ will appear in the controllability matrix at time $t+d,$ and indeed at that time its block at index $t+d+s \ \mbox{mod } p$ will be zero, since $t+s+d =z\  \mbox{mod } p.$ Now, this vector will appear at time $t'>t+d$ in the controllability matrix as $A^{t'-t-d}b.$ This vector has a zero block at index $z+t'-t-d=t'+s\  \mbox{mod } p,$ and hence will never violate the property.  Since we took the time $t$ arbitrarily, no vector in the controllability matrix will \rmj{ever} violate the property.
\end{proof}
The above theorem is easily generalizable to the case where $A$ is a permutation but not a cyclic one, and there is a cyclic subpermutation that satisfies the hypotheses.
One could now conjecture that Theorem \ref{thm-block-permut} actually characterizes all uncontrollable systems: the following example shows that there are more involved situations where the matrices {do not satisfy the hypotheses of the theorem above, but still the system is uncontrollable.  Observe indeed that if one restricts himself to delays smaller or equal to $n$ (like in the theorem above), the system in Example \ref{ex-uncontrollable} is controllable, but uncontrollability can be obtained with large delays.}
\begin{ex}\label{ex-uncontrollable}
Consider a \switcheddelay{} with the following parameters:
$$
A=
\begin{pmatrix}
    \sin\theta_1& -\cos\theta_1&  0&       0\\
    \cos\theta_1& \sin\theta_1&  0&       0\\
    0&       0&       \sin\theta_2& -\cos\theta_2\\
    0&       0&       \cos\theta_2& \sin\theta_2
\end{pmatrix},\
b=
\begin{pmatrix}
1\\ 1\\ 1\\ 1
\end{pmatrix}, \
$$
and $D=\{0,1,\ldots,121\}$, with $\theta_1 = \frac{\pi}{120}$ and $\theta_2 = \frac{\pi}{60}$. Note that $A$ has 2 pairs of complex conjugate eigenvalues $\lambda_1, \lambda_1^*$ and $\lambda_2, \lambda_2^*$ characterized by absolute values $|\lambda_1| = |\lambda_2| = 1$ and phases $\angle(\lambda_1) = \theta_1$, $\angle(\lambda_2) = \theta_2$. As a consequence $\lambda_1$ and $\lambda_1^*$ have equal phases every $80, 160, 240, \ldots$ steps and they have opposite phases every $120, 360, 600, \ldots$ steps. Because of this property the system is uncontrollable: indeed, given the switching signal
$$
\sigma(t) = \left\{
\begin{array}{cl}
  0 & \text{if }\ 0 \leq t \leq 2\\
  121 - t\ mod(121) & \text{if }\ t \geq 3
\end{array}
\right. ,
$$
it is easy to check that the system is uncontrollable, namely $\forall t > 0, rank(C_{t}(A,B,D, \sigma(t))) < 4$. Moreover, \rmj{$\forall 1<t_1<t_2<t_3 < 121, rank([b\ A^{t_1} b\ A^{t_2} b\ A^{t_3} b]) = 4$}.

\rmj{By deriving similar examples with different values of $ \theta_i,$ one can build $4$-dimensional systems which become uncontrollable only if $D$ contains arbitrarily large delays.  This shows that Theorem \ref{thm-block-permut} does not characterize all uncontrollable systems, because this theorem considers systems that are uncontrollable with delays that are smaller than the dimension of the system. }
\end{ex}

Motivated by the above example we provide an algorithm that terminates in finite time for controllability verification (when arbitrarily large look-ahead of the switching signal is allowed). We first prove that the controllability verification problem can always be split into two sub-problems, one characterized by a regular matrix and one by a nihilpotent matrix. Then we provide an exponential time verification algorithm for regular matrices and a polynomial time algorithm for nihilpotent matrices.


\subsection{Problem split into regular and nihilpotent cases}

\begin{lemma}
If the matrix $A$ has more than one Jordan block with eigenvalue zero, the system is uncontrollable.
\end{lemma}
\begin{proof}
This is already the case for systems with constant delay.  It is easy to see that in that case, even the full controllability matrix $[b\  Ab\  A^2b\ \dots \ A^tb]$ cannot be full-rank, even for large $t.$
\end{proof}

\begin{lemma}
Suppose that the matrix $A$ has one Jordan block of size $k$ with eigenvalue zero.  That is, there is an invertible matrix $T$ such that $$ TAT^{-1}=\begin{pmatrix}
	J_{0,k}&0\\0 &A'&
\end{pmatrix},\quad Tb=\begin{pmatrix}
	b_0\\b'
\end{pmatrix}.$$
Then, the system \eqref{sdsystem} is uncontrollable if and only if the system $(J_{0,k},b_0)$ is uncontrollable with the set of delays $D,$ or the system $(A',b')$ is uncontrollable with the set of delays $D.$
\end{lemma}
\begin{proof}
Let us consider the system in the Jordan basis.  If the pair $(A',b')$ is uncontrollable with the set of delays $D,$ then the full system will not be controllable, as the lower block (that is, the last $n-k$ columns, corresponding to the invertible part of $A$) of the controllability matrix will never be of rank $n-k$ {if the switching signal is chosen to be any uncontrollable switching signal for the pair $(A',b').$}\\
Suppose now that the pair $(A',b')$ is controllable with the set of delays $D.$ We {\bf claim} that the whole system is uncontrollable if and only if the pair $(J_{0,k},b_0)$ is {un}controllable with the set of delays $D.$  Indeed, since $A'$ is invertible, if there is a time \rmj{$t^*$ such that the lower block of the controllability matrix is of rank $n-k,$ this will be the case for all $t>t^*.$ Thus, if the pair $(J_{0,k},b_0)$ is controllable with the set of delays $D,$ this means that for any switching signal $\sigma_0(\cdot)$, there will be a time at which the controllability matrix for the pair $(J_{0,k},b_0)$ is full-rank. Since for $t>t^*,$ the upper block of the controllability matrix contains all the columns corresponding to the signal $\sigma_0(t-t^*)=\sigma(t)${, then} there is a time $t>t^*$} at which the upper block, and thus both blocks, of the controllability matrix will be full rank.\\
Conversely, if the pair  $(J_{0,k},b_0)$ is uncontrollable, then there is a switching signal which makes the upper block of the controllability matrix not full rank, and thus the full controllability matrix cannot be full rank either with the same switching signal.
\end{proof}


\subsection{\rmj{The regular case}}

We provide a procedure (Algorithm \ref{algo-controllability}) that {checks in a finite number of steps whether a system (\ref{sdsystem})} with a regular matrix $A$ is controllable and computes the set of all switching signals that make the system uncontrollable. Theorem \ref{thAlg} formally proves the correctness of Algorithm \ref{algo-controllability}. {Our proofs mainly study the dimension of the linear subspace spanned by the controllability matrix, and as a consequence the ideas are generalizable to multiple outputs.} In the following, \rmj{$\begin{pmatrix}n\\k\end{pmatrix}$ denotes the binomial coefficient, as customary}.  \rmj{The main idea of the algorithm is to find a particular subspace $S,$ whose existence is a certificate for uncontrollability.  The definition of subspace $S$ is given in Eq. \eqref{eq-claim-mainthm} below.  We will show in the next theorem that $S$ is actually defined by a \emph{finite} sequence of delays $\sigma=d_1,d_2,\dots, d_{N^*},$ where $N^*$ is an efficiently computable number.  Thus, the algorithm that we present now generates all possible prefixes of such sequences which possibly verify Condition \eqref{eq-claim-mainthm}, until it reaches the bound $N^*.$  It is probably possible to slightly improve the efficiency of the algorithm, but we doubt that it could lead to a polynomial time algorithm.  We defer the study of this question to further work.}

\begin{algorithm}\label{algo-controllability}

\KwData{A triplet $A,b,D$ defining an LTI system with switched delays}
\KwResult{Outputs 'YES' if the system is controllable; outputs 'NO' if the system is uncontrollable\;
\Begin{
\nl Start preprocessing (check {that the matrix is regular})\;
\nl Set $\Sigma := \{\epsilon\}$\;
 \emph{$\Sigma$ is the set of signals possibly extendable to an infinite uncontrollable switching signal. $\epsilon$ is the empty signal. Beware that at that point $\Sigma \neq \emptyset$}\;
\nl $S_\epsilon := \mbox{span}(b)$\;
\nl $t:=0$\;
\nl \While{$\Sigma_t \neq \emptyset$ and $t\leq \begin{pmatrix}n+2|D|\\2|D|\end{pmatrix}$} {
\nl $t:=t+1$\;
\nl $\Sigma_{t} := \emptyset$ \;
\nl \For{each $\sigma \in \Sigma$}{
\nl \For{each $d \in D$}{
\nl \eIf{$A^{-t}b \in (A^{d}S_\sigma)$}{
\nl $\Sigma_{t+1} := \Sigma_{t+1}\bigcup \sigma d$  \quad \emph{($\sigma d$ is the concatenation of $\sigma$ and $d$)}\;
\nl $S_{\sigma d}:=S_{\sigma }$ \;}{
\nl \If{$\spann(S_\sigma \bigcup A^{-d-t}b )\neq \re^n$}{
\nl $\Sigma_{t+1} := \Sigma_{t+1}\bigcup \sigma d$ \;
\nl $S_{\sigma d}:=\spann(S_{\sigma }\bigcup A^{-d-t}b )$\;
}}}}
}
\nl \eIf{$\Sigma_t=\emptyset$}{\nl output 'YES'\;}{\nl output'NO'\;}}}
\caption{The general algorithm}
\end{algorithm}
\normalsize


\begin{theorem}\label{thAlg}
Let $A,b,D$ be respectively a transition matrix, a control input, and a set of delays describing an LTI system with switched delays. {If the matrix $A$ is regular,} Algorithm \ref{algo-controllability} stops in finite time and outputs `YES' if and only if the system is controllable. If the algorithm outputs 'NO`, it also returns the list of all switching signals that make the system uncontrollable.
\end{theorem}
\begin{proof}
The algorithm obviously terminates after at most $\begin{pmatrix}n+2|D|\\2|D|\end{pmatrix}$ steps.\\
We first observe that if the system is uncontrollable, there must exist a nontrivial subspace $S\subset \re^n,\, S\neq \re^n,$ such that \begin{itemize} \item $b\in S,$ \item there exists a switching signal $\sigma(\cdot{})$ satisfying \begin{equation}\label{eq-S}S=\spann(C_t(A,b,D,\sigma(\cdot)));\end{equation} \item there is no $t'>t$ such that $\spann(C_{t'}(A,b,D,\sigma(\cdot)))$ has a larger dimension than $S$.\end{itemize}  Indeed, if $\sigma(\cdot)$ is the uncontrollable switching signal, there is no $t$ such that $\spann(C_t(A,b,D,\sigma(t)))=\re^n,$ meaning that the dimension of this space must reach a maximum integer smaller than $n$ for some $t.$ Also, from the definition of $C_t,$ we have that $b$ is the last column of $C_t$ at every time $t$ such that $t=t'+\sigma(t')$ for a certain $t'<t.$ Thus, for these $t,$ $b\in S.$

It turns out that one can say much more about $S:$ We {\bf claim} that there is such a subspace $S$ such that for all $s\in \n,$ there exists a particular delay in $D,$ which we note $d'(s)\in D,$ such that $$A^{-d'(s)-s}b\in S. $$  Indeed, let us fix $t$ such that Equation (\ref{eq-S}) is satisfied for our maximal-dimension subspace $S.$ Observe that for any $t'>0,$ $$C_{t+t'}=[A^{t'}C_t(A,b,D,\sigma(\cdot))| C_{t'}(A,b,D,\sigma(\cdot+t))] .$$  Since $S$ has the largest possible dimension and $A$ is regular, $\spann(C_{t+t'})=\spann(A^{t'}C_t)=A^{t'}S,$ and thus for all columns $c$ of $ C_{t'}(A,b,D,\sigma(\cdot+t)),$ we have $c \in A^{t'}S.$ In particular, for $t''\leq t'-d_{max},$ we have a column \begin{equation}\label{eq-column}c=A^{t'-t''-\sigma(t+t'')}b\in A^{t'}S.\end{equation} Recall that we have chosen $t'$ arbitrarily. Thus, for any $t''>0,$ we can take $t'>t''+d_{max},$ so that (\ref{eq-column}) holds, and this implies by regularity of $A$ that \begin{equation}\label{eq-claim-mainthm} \forall s>0,\exists d'(s)\in D: \, A^{-s-d'(t+s)}b\in S, \end{equation} which was the claim.

Equation \eqref{eq-claim-mainthm} allows us to define a signal $\sigma^*(s)=d(s)$ which makes the system uncontrollable.  To see this, consider an arbitrary time $t',$ and remark that any column of the controllability matrix satisfies the following equation $$c\in C_t \rightarrow \exists s>0,s+\sigma^*(s)\leq t, \, c=A^{t-s-\sigma^*(s)}b.$$  The last equation together with \eqref{eq-claim-mainthm} imply that all columns of $C_t$ are in $A^tS,$ and thus they cannot span $\re^n.$\\ Thus, Equation \eqref{eq-claim-mainthm} is satisfied if and only if the system is uncontrollable.  Our algorithm simply tests this condition for increasing $s$ up to $$N^*:=\begin{pmatrix}n+2|D|\\2|D|\end{pmatrix}.$$  \rmj{However, one does not know the space $S$ before the algorithm terminates, and thus the algorithm explores all the different possible spaces.} Hence, it remains to show that if Equation (\ref{eq-claim-mainthm}) has a solution for $s=1,\dots,N^*,$ then it has a solution for all $s>0.$

Let us suppose that \eqref{eq-claim-mainthm} has a solution $(S,\sigma(\cdot)),$ and define a polynomial \rmj{$p(x)$ of minimal degree such that} $$ p(x): \re^n\rightarrow \re, \quad p(x)=0 \Leftrightarrow x\in \bigcup_{d\in D}{A^dS}.$$  This polynomial has degree smaller or equal\footnote{Indeed, a linear subspace $S=\{x\in \re^n:\, c_i^Tx=0,\ 1\leq i <n\}$ can be written as the roots of a polynomial $S=\{x\in \re^n:\, \sum{(c_i^Tx)^2} =0,\ 1\leq i <n\},$ and the union of two linear subspaces $S_1,S_2$ defined accordingly by two polynomials $p_1, p_2$ is given by $S_1\bigcup S_2 =\{x\in \re^n:\, p_1(x)p_2(x)=0\}.$} to $2|D|$. Now, $b\in S$ implies that $p(b)=0,$ and by regularity of $A,$ condition (\ref{eq-claim-mainthm}) restricted to the first $N^*$ values of $s$ can be rewritten as $$p(A^{-s}b)=0\quad s=1,\dots,N^*.$$ Looking now at $p(x)$ as an element of the $N^*$-dimensional vector space  of all the polynomials of degree $2|D|$ in $n$ variables, we have that $p(x),p(A^{-1}x),\dots,p(A^{-N^*}x)\in V, $ where $V$ is the linear subspace of all the polynomials such that $p(b)=0.$ Since the application $$L_{A}:\quad p(x)\rightarrow p(A^{-1}x) $$ is a linear application on \rmj{that} $N^*$-dimensional vector space, we have $$\forall s\in[0,N^*],\ L_{A}^s(p)\in V \Leftrightarrow \forall s\geq 0,\ L_{A}^s(p)\in V.$$ That is, $$ \forall s\geq 0, p(A^{-s}x) \in V,$$ that is, $$\forall s\geq 0, A^{-s}b\in \bigcup_{d\in D}{A^dS}.$$ This last equation is equivalent to Equation \eqref{eq-claim-mainthm}, and this concludes the proof.
\end{proof}

\begin{corollary}\label{corMinimumN}
Given a controllable system \eqref{sdsystem} with $m=1,$ if the look-ahead satisfies $N \geq \small \begin{pmatrix} n+2|D|\\2|D|\end{pmatrix}, \normalsize$ one can drive the system trajectory for any initial condition $x_0$ and any switching signal $\sigma(t)$ to any arbitrary final state using the controller as in Definition \ref{def-dep}. Moreover, it is possible to reach the final state in the worst case in less than $\small \begin{pmatrix} n+2|D|\\2|D|\end{pmatrix} \normalsize$ time steps.
\end{corollary}
\begin{proof}
Since \eqref{sdsystem} is controllable and $m=1$, by Theorem \ref{thAlg}, $C_t$ has full rank for all $t \geq \small \begin{pmatrix} n+2|D|\\2|D|\end{pmatrix} \normalsize$ for any switching signal $\sigma(t)$, and by Equation \eqref{eqControllability} the result follows.
\end{proof}


\subsection{A polynomial time algorithm for the nihilpotent case}

For the nihilpotent case we are thus left with the problem of deciding whether a system with a controllable pair $(J_{0,k},b)$ is uncontrollable with a given set of delays $D.$  It turns out that there is a very simple combinatorial characterization of controllability in that case:

\begin{lemma}
Suppose that the matrix $A$ is a single Jordan block of size $k$ with eigenvalue zero. Then, the controllability matrix is full rank if and only if $b,Ab,A^2b,\dots A^{k-1}b$ are columns of the controllability matrix.
\end{lemma}
\begin{proof}
This is straightforward from the fact that $(A,b)$ is controllable in the classical sense, but for all $t \geq k,$ $A^t=0.$
\end{proof}

Thus, the controllability problem for nihilpotent matrices amounts to check whether there exists a sequence of delays $\sigma:\, \n \rightarrow D$ such that the controllability matrix never contains  $b, Ab, A^2b, \dots, A^{k-1}b$. \rmj{ Stated otherwise, a system is controllable if for every switching signal, the corresponding signal of \emph{actuation times} (see Definition \ref{def-general-sdsystem}) contains $k$ '$1$' in a row.} 

In this subsection we derive an efficient (polynomial time) procedure for checking this property. 
 In fact, we are able to characterize the uncontrollable systems as follows:

\begin{theorem}
Let $A,b,D$ represent a system as in \eqref{sdsystem}. If $A$ is a nihilpotent matrix, then the system is uncontrollable, except if $A$ is (similar to) a $2\times 2$ Jordan block and the set of delays has only two different delays, of equal parity. In that case the minimal look-ahead needed to guarantee controllability is equal to $d_{max}.$
\end{theorem}
\begin{proof}
We first show that the system is controllable in the particular case where $A$ is similar to a $2 \times 2$ Jordan block and \rmj{both} delays in $D$ have the same parity. In all the other cases we exhibit an uncontrollable switching signal, namely a switching signal that makes the system uncontrollable.

We recall that if $A$ is (similar to) a single Jordan block of size $k$ with eigenvalue zero, the controllability matrix is of rank $k$ if and only if it contains $b,Ab,\dots, A^{k-1}b$.

\rmj{This implies that if $A$ is $2\times 2$ and one wants to build an uncontrollable switching signal,} there cannot be two consecutive equal values for the switching delay.  Indeed, if $\sigma(t)=\sigma(t+1)=d,$ then, at time $t+1+d, $ the controllability matrix contains both vectors $b$ and $Ab.$  This makes the controllability matrix of rank two since the pair $A,b$ is controllable.  Thus, the only potentially uncontrollable switching signal is of the shape
$$ \sigma(t) = d_1d_2d_1d_2\dots $$  However, if $D=\{d_1,d_2\},$ with both delays of the same parity (we take $d_1<d_2$ without loss of generality), at time $t=2+d_2,$ the controllability matrix contains $b$ (because $\sigma(2)=d_2$) and $Ab$ (because $\sigma(t+d_2-1-d_1)=d_1$), and the signal is not uncontrollable.

If $A$ is $2\times 2$ and there is an odd delay $d_1$ and an even delay $d_2,$ then the signal $$\sigma(t) = d_1d_2d_1d_2\dots$$ is uncontrollable.  Indeed, for each time $t,$ the controllability matrix receives the vector $b$ at an even time $t+\sigma(t).$ As a consequence, there are no two consecutive times $t,t+1$ at which the controllability matrix contains the vector $b,$ and thus it never contains the pair $b,Ab.$

If $A$ is $2\times 2$ and there are (at least) three different even delays: $D=\{d_1,d_2,d_3,\dots\},\ d_1,d_2,d_3= 0\mbox{ mod }2,$ there is also an uncontrollable switching signal.  For the sake of clarity, we first suppose that $d_1=0.$ This does not incur a loss of generality, because the set of delays $D'=\{0,d_2-d_1,d_3-d_1\}$ is controllable if and only if the initial set of delays is. Indeed, any uncontrollable switching signal on $D$ (resp. $D'$) immediately translates to an uncontrollable signal on $D'$ (resp. $D$) by taking the corresponding delay in the other set at every step (the signal of actuation times is simply the same, shifted by $d_1$ values).\\ Now, we suppose $d_1=0$ and we define an uncontrollable periodic switching signal depending only on the ratio between $d_2$ and $d_3.$\\
If $d_3\leq 2d_2,$ take the signal of actuation times equal to $$
\tau = (001(01)^{d_2/2-1})^*,$$ where $(w)^l$  means the concatenation of $l$ times the word \rmj{$w,$} and $w^*$ means the infinite repetition of the word $w.$  Thus, $\tau$ is a periodic signal of period $d_2+1.$  The signal $\tau$ is uncontrollable because there are never two consecutive ones. We claim that there is a switching signal $\sigma(t)$ corresponding to this signal of actuation times, that is, such that $\forall t, \tau(t+\sigma(t))=1.$  Since $\tau$ is periodic, it is sufficient to prove it for the first period. The important property of $\tau$ is that for the first period, all odd times have value one (except the first one), for the second period, all even times have value one (except the first one), etc.\\  For all the times $t$ such that $\tau(t)=1,$ one can take $\sigma(t)=d_1=0.$ For $t=1,$ one can take $\sigma(1)=d_2.$ Indeed, one sees in the equation above that the $1+d_2$-th value of $\tau$ (the last digit of the period) is equal to one.  Now, take $\sigma(2)=d_3,$ and for all other even times, take $\sigma(t)=d_2.$  For this choice, $t+\sigma(t)$ will be an even number in the second period (but not its very first time), thus equal to one (recall that $d_2<d_3\leq 2d_2$).\\
If $d_3>2d_2,$ take $$\tau = 001((01)^{(d_3-d_2-2)/2})^*.$$
Now the period is equal to $d_3-d_2+1.$  Again, in the first period all the odd times (but the first one) are equal to one, in the second period all the even times (but the first one), etc.\rmj{, and we focus on the first period.}  Of course, we take $\sigma(t)=d_1=0$ for the odd times, except for $t=1.$  For $t=1,$ we take $\sigma(t)=d_2.$ Since $1+d_2\leq d_3-d_2 +1,$ this is an odd time in the first period, for which the signal $\tau$ is equal to one.  For the even times $2\leq t\leq d_3-2d_2 +2,$ take $\sigma(t)=d_3.$ Since $d_3+2\leq t+d_3 \leq 2 (d_3-d_2 +1)$ is an even time in the second period (but not the first time of the second period), the signal $\tau$ is equal to one. \\Now, for the even times $ t> d_3-2d_2 +2,$ take $\sigma(t)=d_2.$  Since $t+d_2>d_3-d_2+2,$ this time is an even time in the second period (and not the very first one), and thus again the signal $\tau$ is equal to one.

Finally, if $A$ has dimension larger than $2,$ we claim that one can build the uncontrollable signal starting at $t=1$ and
defining $\sigma$ in a greedy way for increasing $t:$ always take
$\sigma(t)=d_1,$ except if this makes the signal controllable, in which
case take $\sigma(t)=d_2.$  More precisely, the algorithm builds a signal of {actuation times} $\tau(t)$ in the following way: start at $t=1$ with the zero signal (i.e., $\forall t\geq 1, \tau(t)=0$), and define $\sigma(t)=d_1$ (and thus $\tau(t+d_1)=1$) if it does not incur a sequence of $k$ consecutive ones in $\tau.$  In the opposite case, define $\sigma(t)=d_2$ (and thus $\tau(t+d_2)=1$).\\
This algorithm builds a signal of actuation times that never contains $k$ ones in a row. We prove this by contradiction: suppose that at some time $t$ the above described algorithm creates a sequence of $k$ consecutive ones in the signal $\tau,$ and let us take the first time $t$ at which this happens. Then, the consecutive sequence occurs at the places $\tau(t+d_2-k+1)\dots \tau(t+d_2)$: {this is because by design, the algorithm does not chose $d_1$ if it creates a sequence of consecutive ones, so the problem must occur when the algorithm uses the delay $d_2.$ Moreover, looking to the signal $\tau$ as it is when the problem occurs, we have $\tau(t')=0$ for all $t'> t+d_2.$} Thus, we have $\tau(t+d_2)=\tau(t+d_2-1)=\tau(t+d_2-2)=\dots =\tau(t+d_2-k+1)=1.$ This implies in turn that $\tau(t+d_1)=\tau(t+d_1-1)=\dots =\tau(t+d_1-k+1)=0$ (because if $\tau(t+d_1-k')$ was equal to one for any value $0\leq k' \leq k-1$, the algorithm would have picked $\sigma (t-k')=d_1$ without doing any harm).  However, this is impossible, because $\sigma(t-1)$ can safely be put to $d_1$ in that case; and we have reached a contradiction.
\end{proof}

\begin{corollary}
If $A$ is a nihilpotent matrix, there is a polynomial time algorithm to decide whether System \eqref{sdsystem} is controllable with arbitrarily large look-ahead.
\end{corollary}

\begin{remark}
In practice our algorithm for nihilpotent matrices is a constant time algorithm: we basically provide a characterization of the pathological uncontrollable situations, and the algorithm can check these conditions in essentially constant time, independently of the state space dimension or the precise values of the delays.
\end{remark}

The following corollary is directly implied by the above corollary and Theorem \ref{thAlg}.
\begin{corollary}
Given any $A,b,D$ there is an algorithm to decide in finite time whether the system is controllable with arbitrarily large look-ahead.
\end{corollary}


\section{\ad{Controller design with fixed look-ahead}} \label{secDesignDI}

When the controller cannot chose the look-ahead the design problem is trickier. In particular, when the look-ahead is zero, one has to design a single controller that would work for any possible switching signal. In this section we make some initial steps to tackle the design problem in the delay independent case. We first provide an example of a very simple scalar system where, according to the dynamics of the plant, the system can be not-stabilizable, stabilizable with memory, or stabilizable without memory.

\begin{ex}\label{ex-indep}
In this example \rmj{we consider the simplest nontrivial situation, namely an arbitrary one dimensional linear system with delays in $ D=\{0,1 \}.$} We assume that the controller stores the previous value of $x(t)$ instead of the previous value of $v(t).$  We make this choice for the sake of clarity, in order to have simpler matrices in the equivalent switching system. It is easy to check that in this slightly modified setting, one can still apply the trick of Theorem \ref{thm-ssreduction-independent} and obtain the following three-dimensional switching system.
\begin{equation}\label{eq-ssreduction-independent}\begin{pmatrix}x({t})\\ x({t+1})\\ u_1(t+1)\end{pmatrix}= M_{\sigma(t)} \begin{pmatrix}x({t-1})\\ x(t)\\ u_1(t)\end{pmatrix}.\end{equation}

where
$$ M_0=\begin{pmatrix}0& 1 &0\\ bk_1 &a+ bk_2& 1\\0& 0 &0 \end{pmatrix}, \quad M_1=\begin{pmatrix}0 &1& 0\\ 0& a& 1\\ bk_1 &bk_2 &0\end{pmatrix}. $$

The controller stores the value of $x({t-1})$ for one iteration. It then makes use of it and of the current value $x(t)$ for computing its output $v(t).$ If the delay is 1 $v(t)$ is put "in the queue" (third entry of the vector), while if the delay is zero it is directly added in the plant in order to compute $x({t+1}).$ Let us fix $b=1.$ Depending on the other values, we obtain the following cases:
\begin{itemize}
\item For $a>3,$ system (\ref{eq-ssreduction-independent}) is unstable, whatever controller $K$ is applied. This can be seen by observing that in this case, $trace(M_1)\geq 3,$ hence $M_1$ is unstable. By Theorem \ref{thm-ssreduction-independent}, the corresponding \switcheddelay{} is uncontrollable.
\item For $a<1$ the system is clearly stabilizable without any controller ($k_1=k_2=0$), since it is the case for the autonomous stable dynamical system.
\item For a=1.1 the system is controllable without using memory (i.e. $k_1=0$), e.g. by taking $k_2=-0.5,$ and the switching system (\ref{eq-ssreduction-independent}) is stable.
\item Finally, for $a=2$ the system is still controllable, but in this case one needs $k_1\neq 0$: indeed, if $k_1=0$ one can restrict himself to the $2 \times 2$ lower-right corner of the matrices, and this subsystem is unstable because $trace(M_0)\geq 2$. On the other hand stabilizing controllers exist with $k_1\neq 0,$ as for instance $k_1=0.4,k_2=-1.5.$
\end{itemize}
In the last two cases, in order to check for the validity of the proposed controller, one can check that the joint spectral radius of the set $\{M_0,M_1\}$ corresponding to the proposed controller is smaller than one, for instance by making use of the JSR toolbox, available on the web \cite{VHJ_HSCC14}.
\end{ex}

We now present an example for which there is a nonlinear controller which works better than any linear controller: an asset of our nonlinear controller is that it can detect at time $t+1$ the switching signal $\sigma(t)$ using the state-space measure $x(t+1)$ as a proxy, and make use of it in order to improve the next control signals.

\begin{ex}\label{ex-nonlin}
The system in this example is a rotation of an angle $\alpha:$ $$ A=\begin{pmatrix}
	\ca&-\sa\\ \sa & \ca
\end{pmatrix},\quad B = \begin{pmatrix}
	1\\ 0
\end{pmatrix}.$$ The set of delays is $D=\{0,1\},$ and the controller is delay-independent.
We show in the next theorem that this system is stabilizable with a rate of decay equal to $0.69\dots,$ but no linear controller can achieve a rate smaller than $0.755\dots.$
\end{ex}

\begin{theorem}
For the system from Example \ref{ex-nonlin}, there are values of $\alpha$ such that if the system is stabilized with a linear controller, one cannot guarantee the existence of a constant $K$ such that $$x(t)\leq K \rho^t $$ for any $\rho <0.755\dots.$   However, there is a nonlinear controller that allows such a guarantee with $\rho=0.69\dots.$
%
\end{theorem}

\begin{proof}
We suppose $\alpha$ relatively small.  The nonlinear controller that we propose relies on the fact that if a point is close to the line of slope $-\alpha\frac{3}{2}$ which contains the origin, then, within three steps, it can be mapped on a new point close to this line, whose norm is approximately three times smaller than the initial vector.  We denote this line by $L.$  Then, by iterating this argument, one gets a controller whose rate of convergence towards the origin is close to $(1/3)^{1/3}\approx .69\dots.$ We will then show that no linear controller can reach a rate of convergence smaller than $.7555\dots .$ We {\bf claim} that if the point $x(t)$ satisfies $x=(w/\tan{(-3\alpha/2)},w)+r,$ for some real number $w,$ and some vector $r$ such that $||r||\leq 2|w|\sin{(\alpha)},$ then, one can chose a control sequence $v(t),v(t+1)$ (and potentially $v(t+2)=0$) such that $x(t+2)=(w'/\tan{(-3\alpha/2)},w')+r',$ (or $x(t+3)=(w'/\tan{(-3\alpha/2)},w')+r'$) for some real number $w$ such that $|w'|\leq |w|/3$ and some vector $r'$ such that $||r'||\leq 2 |w'|\sin{(\alpha)}.$  The idea is that after one step, $x(t)$ will rotate very close to the $x$-axis (at an angle approximatively $\alpha/2$), and the controller will then project it close to the $y$-axis, approximatively multiplying its norm by $\sin{(\alpha/2)}.$  One will then add a horizontal component to the new vector in order to remap it on the line $L.$  Due to the possible delay, this two step method cannot be implemented in less than three actual steps.  Also, one must pay attention to the errors incurred by the unpredictable delays occuring in the process (represented by the vector $r$).  We now describe the control signal, and compute the corresponding evolution of the system.

We suppose that $x(t)-r$ (i.e., the point on $L$ closest to $x(t)$) lies in the second orthant; the case for the fourth orthant is exactly symmetrical.
We define the actuation signal $v(t)=||x(t)-r||\cos{(\alpha/2)}.$ Remark that $v(t)$ is the control signal that would exactly project the point $Ax(t),$ and also the point $A^2x(t),$ on the $x$-axis if $x(t)$ was exactly on the line.  Now, depending on the delay $\sigma(t),$ two cases can occur:
\begin{eqnarray}\label{eq-ex-1}
x(t+1)&= &\begin{pmatrix}
0\\	||x(t)-r||\sin{(\alpha/2)}
\end{pmatrix} +Ar\quad \mbox{if } \sigma(t)=0,\\\label{eq-ex-2}
x(t+1)&= &\begin{pmatrix}
-||x(t)-r||\cos{(\alpha/2)}\\	||x(t)-r||\sin{(\alpha/2)}
\end{pmatrix} +Ar,\quad \mbox{if } \sigma(t)=1.
\end{eqnarray}
\normalsize
Now, we define $\hat x(t+2)$ to be the vector at time $t+2$ if $u(t+1)=0$. That is, in the first case (i.e. $\sigma(t)=0$) we define
\begin{equation}\label{eq-ex-hat1}
\hat x(t+2)=Ax(t+1),
\end{equation}
in the second case (i.e. $\sigma(t)=1$),
\begin{equation}\label{eq-ex-hat2}
\hat x(t+2)=Ax(t+1)+v(t).
\end{equation}
We call $w'=\mbox{ord}(\hat x(t+2))$ and we define $v(t+1)=-\mbox{abs}(\hat x(t+2))-w'/\tan{(3\alpha/2)}$, where \lq $\mbox{ord}$\rq{} and \lq $\mbox{abs}$\rq{} denote respectively the second and first coordinate of a vector. This is the control signal which, if applied at time $t+2,$ maps $x(t+2)$ onto the line $L,$ on the point $(w'/\tan{(-3\alpha/2),w')}.$ We will prove below that if $\sigma(t+1)=0,$ then $x(t+2)$ satisfies the claim.  Now, if $x(t+2)$ does not satisfy the claim, we define the control signal $v(t+2)=0,$ and we will show that in this case $x(t+3)$ satisfies the claim.

Let us first suppose that $\sigma(t+1)=0.$ Then, the point $x(t+2)$ is exactly on the line $L,$ at $(-w'/\tan{(3\alpha/2)},w'),$ where $w'=\mbox{ord}(\hat x(t+2)).$
We can compute $w'$ thanks to equations (\ref{eq-ex-1}-\ref{eq-ex-hat2}) for both values of the delay $\sigma(t).$
\begin{eqnarray}\label{eq-ex-1bis}
w'&= &||x(t)-r||\sin{(\alpha/2)}\ca + \mbox{ord}(A^2r)\quad \mbox{if } \sigma(t)=0,\\ \label{eq-ex-2bis}
w'&= &-||x(t)-r||\sin{(\alpha/2)}+ \mbox{ord}(A^2r)\quad \mbox{if } \sigma(t)=1.
\end{eqnarray}

Since $||x(t)-r||=w/\sin{(3\alpha/2)}$ and $||r||\leq 2 w \sa,$ we have that $|w'|<|w|(1/3+\epsilon),$ where $\epsilon$ can be taken arbitrarily small (by taking a sufficiently small $\alpha$).\\
We now analyze the case where $\sigma(t+1)=1.$ In this case, $x(t+3)=A\hat x(t+2)+v(t+1)=\hat x(t+2)+v(t+1) + (A-I)\hat x(t+2).$  We have just shown above that $\hat x(t+2)+v(t+1) $ is a vector $(-w'/\tan{(3\alpha/2)},w')\in L$ satisfying $|w'|<|w|(1/3+\epsilon),$ so, we only have to show that $r'=(A-I)\hat x(t+2)$ satisfies $||r'||<2 \sa w',$ which we show by the submultiplicativity of the Euclidean norm: First, we have $||(A-I)||\approx \alpha$ for small values of $\alpha.$ Second, if $\sigma(t)=0,$ we have (see again (\ref{eq-ex-1}-\ref{eq-ex-hat2}))
$$||\hat x(t+2)||=||x(t+1)||\leq ||x(t)-r||\sin{(\alpha/2)}+ ||Ar||< 2|w'|,$$
\normalsize
and if $\sigma(t)=1,$ we have
$$||\hat x(t+2)||\leq ||x(t)-r||\sin{(\alpha/2)}+ ||A^2r||< 2|w'|,$$
\normalsize
which proves the claim (for $\alpha$ sufficiently small).\\
In order to implement the controller, one still has to fulfill the inductive hypothesis of the claim at the first step, that is, obtain a point which is close to the line $L.$ This is easily obtained by defining $v(0)=-\mbox{abs}(Ax(0))-\mbox{ord}(Ax(0))/\tan{(3\alpha/2)}$ (this is exactly identical to the way we compute $u(t+1)$ from $\hat x(t+2)$ above, and we skip the details).

It remains to show that no linear controller can achieve a decay rate better than $0.75.$ {It is not difficult to see that using a linear controller, the closed loop system can be modeled as the following switching system (see Proposition 1 in \cite{JungersCDC2012} for details):}
\begin{equation}
\{A_0,A_1\} =  \left\{ \begin{pmatrix}
	\ca +k_1& -\sa + k_2 & 1\\ \sa & \ca & 0 \\ 0&0&0
\end{pmatrix}, \begin{pmatrix}
	\ca & -\sa & 1\\ \sa & \ca & 0 \\ k_1 & k_2 & 0	
\end{pmatrix}\right \}. 
\label{eq-thm-nonlin-matr}
\end{equation}
\normalsize
We show that these two matrices cannot both have a spectral radius smaller than $0.755, $ and thus no linear controller can ensure a convergence to zero at a rate smaller than $0.755.$
The characteristic polynomial of the first matrix is $$\lambda(\lambda^2- (k_1+2\ca) \lambda +1 +k_1 \ca -k_2 \sa), $$
%
and thus its two nonzero eigenvalues satisfy $$\lambda_1\lambda_2= 1 +k_1 \ca -k_2 \sa. $$
Thus, $|\lambda_1\lambda_2|\leq \rho^2$ implies
\begin{eqnarray}
1 +k_1 \ca -k_2 \sa \leq \rho^2\ \Leftrightarrow\ k_2 \sa -k_1 \ca \geq 1- \rho^2. \label{eq-cond-stab}
\end{eqnarray}
Now, the determinant of $A_1$ is equal to $k_2\sa - k_1 \ca,$ and since the determinant is the product of the eigenvalues, the condition that the norm of the eigenvalues of $A_1$ have modulus smaller than $\rho$ implies \begin{equation}
k_2\sa - k_1 \ca \leq \rho^3.
\label{eq-cond-stab2}
\end{equation}
Putting (\ref{eq-cond-stab}) and (\ref{eq-cond-stab2}) together yields $ \rho^2 +\rho^3 \geq 1, $ and thus $ \rho \geq 0.755, $ which concludes the proof.

\end{proof}

The above theorem implies that restricting oneself to linear controllers leads to conservativeness. This is, to our knowledge, the first example of that kind. It shows that, contrary to what has been done until now in the switching systems literature, one should not restrict his search to linear controllers in order to avoid conservativeness.


\section{Conclusion}\label{secConclusions}

Motivated by applications in wireless control networks, we start in this paper a research line where we introduce and analyze a new model of linear time invariant (LTI) systems with switching delays. Switching delays have indeed attracted attention in recent years in the control community, because they appear naturally in several modern technological systems, like WCNs, computer networks, etc.

These systems can be represented as particular switching systems, but in terms of difficulty, they seem to lie halfway between easy LTI systems (for which closed form formulas or efficient algorithms are available for most control questions) and switching systems (for which even the simplest questions are intractable, or bound to conservative solutions). Contrary to what has been done in the recent literature, our goal was not to apply or adapt switching systems-oriented techniques to our systems, but rather to develop methods that would take advantage of the algebraic structure of these particular switching systems.

We have provided an analytical characterization of controllable systems when enough delays are known in advance. This condition is more complex than for LTI systems, but yet, allows for an algebraic characterization of controllable systems. On the contrary, we have shown that in general, unlike LTI systems, ours can need nonlinear strategies in order to be stabilized, if there is a limit on the number of forthcoming delays that the controller can know in advance.

We believe that this type of results will be useful in forthcoming applications, where the design of the controller and of the communication protocol (i.e. the routing policy) are tangled, as for instance in a multi-hop networked control system. On top of their practical usefulness, we believe that the questions we investigate are promising for theoretical research, as they necessitate new theoretical ideas in order to answer classical control questions, but yet, they seem to allow for closed form characterizations and algorithmic decision procedures (as opposed to Lyapunov-like methods, which are prevalent for switching systems and are often bound to conservativeness and lack of structure).

\rmj{We have introduced a model here which slightly differs from recent works on systems with switching delays and WCNs in that the delay is determined at the time where the controller emits a control packet, rather than at the time where the plant receives the control packet. This seemed to us closer to practical situations.  We believe that our approach is transposable to different models with switching delays, and that similar phenomena occur, independently of the details of the model.  We leave this question for future research.}

We raise several open problems, for instance: Is the stability analysis for a delay-independent system (i.e. $N=0$) easier than in the general case (which we proved NP-hard)?  Is there a more efficient version of Algorithm 1, which would allow to decide controllability of a delay-dependent system (with sufficient look-ahead) in polynomial time?  When the look-ahead is bounded (or even zero), how to efficiently design a non-linear controller? This task seems hard, especially in view of the fact that the controller can be quite intricate (like in Example \ref{ex-nonlin}). We believe that many more interesting results remain to be unravelled about linear systems with switching delays.


\def\cprime{$'$} \newcommand{\noopsort}[1]{} \newcommand{\singleletter}[1]{#1}


\end{document}